\documentclass{amsart}
\usepackage{amsmath,amssymb,amsthm,amsfonts,amscd,mathrsfs}
\usepackage{tikz}
\usepackage{braids}
\usepackage[all]{xy}
\usepackage[hyphens]{url}

\theoremstyle{plain}
    \newtheorem{thm}{Theorem}[section]
    \newtheorem{lem}[thm]   {Lemma}
    \newtheorem{cor}[thm]   {Corollary}
    \newtheorem{prop}[thm]  {Proposition}

\theoremstyle{definition}
    \newtheorem{defn}[thm]  {Definition}

    \newtheorem{rem}[thm]{Remark}

\def\max{\mathrm{max}}

\def\cat{\mathsf{cat}}
\def\secat{\mathsf{secat}}

\def\dim{\mathrm{dim}}
\def\ker{\mathrm{ker}}
\def\nil{\mathrm{nil}}

\newcommand{\be}{\begin{enumerate}}
\newcommand{\ee}{\end{enumerate}}
\newcommand{\R}{\mathbb{R}}
\newcommand{\Z}{\mathbb{Z}}
\newcommand{\C}{\mathbb{C}}
\newcommand{\Q}{\mathbb{Q}}

\newcommand{\co}{\colon\thinspace}

\newcommand{\PB}{\mathcal{P}}
\newcommand{\TC}{{\sf TC}}

\newcommand{\chd}{\mathrm{cd}}

\begin{document}

\title[Topological complexity of aspherical spaces]{New lower bounds for the topological complexity of aspherical spaces}

\author{Mark Grant}
\author{Gregory Lupton}
\author{John Oprea}

\address{School of Mathematics \& Statistics,
Herschel Building,
Newcastle University,
Newcastle upon Tyne NE1 7RU,
UK}

\email{mark.grant@nottingham.ac.uk}

\address{Department of Mathematics, Cleveland State University, Cleveland OH 44115 U.S.A.}

\email{g.lupton@csuohio.edu}
\email{j.oprea@csuohio.edu}

\date{\today}


\begin{abstract} We show that the topological complexity of an aspherical space $X$ is bounded below by the cohomological dimension of the direct product $A\times B$, whenever $A$ and $B$ are subgroups of $\pi_1(X)$ whose conjugates intersect trivially. For instance, this assumption is satisfied whenever $A$ and $B$ are complementary subgroups of $\pi_1(X)$. This gives computable lower bounds for the topological complexity of many groups of interest (including semidirect products, pure braid groups, certain link groups, and Higman's acyclic four-generator group), which in some cases improve upon the standard lower bounds in terms of zero-divisors cup-length. Our results illustrate an intimate relationship between the topological complexity of an aspherical space and the subgroup structure of its fundamental group.
\end{abstract}

\thanks{This work was partially supported by grants from the Simons Foundation: (\#209575 to Gregory Lupton)
and (\#244393 to John Oprea).}

\maketitle
\section{Introduction}\label{sec:intro}

Topological complexity is a numerical homotopy invariant introduced by Farber in the articles \cite{Far03,Far04}. As well as being of intrinsic interest to homotopy theorists, its study is motivated by topological aspects of the motion planning problem in robotics. Define the topological complexity of a space $X$, denoted $\TC(X)$, to be the sectional category of the free path fibration $\pi_X\co X^I\to X\times X$, which sends a path $\gamma$ in $X$ to its pair $(\gamma(0),\gamma(1))$ of initial and final points. The number $\TC(X)$ gives a quantitative measure of the `navigational complexity' of $X$, when viewed as the configuration space of a mechanical system. Topological complexity is a close relative of the Lusternik--Schnirelmann category $\cat(X)$, although the two are independent. Further details and full definitions will be given in Section 2.

We remark once and for all that in this paper we adopt the convention of normalizing all category-type invariants to be one less than the number of open sets in the cover. So for instance, $\TC(X)=\cat(X)=0$ when $X$ is contractible.

Recall that a path-connected space $X$ is \emph{aspherical} if $\pi_i(X)=0$ for $i\ge 2$. The homotopy type of an aspherical space is determined by the isomorphism class of its fundamental group. Furthermore, for any discrete group $G$ one may construct, in a functorial way, a based aspherical complex $K(G,1)$ having $G$ as its fundamental group. Through this construction, any new homotopy invariant of spaces leads to a new and potentially interesting algebraic invariant of groups. In this paper we address the following problem, posed by Farber in \cite{Far06}: can one express $\TC(G):=\TC\big(K(G,1)\big)$ in terms of algebraic properties of the group $G$? This is an interesting open problem, about which relatively little is known beyond some particular cases (see below). In contrast, the corresponding problem for Lusternik--Schnirelmann category was solved in the late 1950's and early 1960's, with work of Eilenberg--Ganea \cite{EG}, Stallings \cite{Sta} and Swan \cite{Swa}. Their combined work showed that $\cat(G):=\cat(K(G,1))=\chd(G)$, where $\chd$ denotes the \emph{cohomological dimension}, a familiar algebraic invariant of discrete groups.

Groups $G$ for which the precise value of $\TC(G)$ is known include: orientable surface groups \cite{Far03}; pure braid groups $\PB_n$ \cite{FY} and certain of their subgroups $\PB_{n,m}=\ker (\PB_n\to \PB_m)$ which are kernels of homomorphisms obtained by forgetting strands \cite{FGY}; right-angled Artin groups \cite{C-P08}; basis-conjugating automorphism groups of free groups \cite{C-P08-2}; and almost-direct products of free groups \cite{C10}. In all of these calculations, sharp lower bounds are given by cohomology. If $k$ is a field, let $\cup\;\co H^\ast(G;k)\otimes H^\ast(G;k)\to H^\ast(G;k)$ denote multiplication in the cohomology $k$-algebra of the group $G$. The ideal $\ker (\cup)\subseteq H^\ast(G;k)\otimes H^\ast(G;k)$ is called the \emph{ideal of zero-divisors}. One then has that $\TC(G)\ge \operatorname{nil} \ker (\cup)$, where $\nil$ denotes the nilpotency of an ideal. This is often referred to as the \emph{zero-divisors cup-length} lower bound.

On the other hand, it is known that zero-divisors in cohomology with untwisted coefficients are not always sufficient to determine topological complexity. In \cite{G1} the topological complexity of the link complement of the Borromean rings was studied, and sectional category weight and Massey products were applied to obtain lower bounds. This is, to the best of our knowledge, the only previously known example of an aspherical space $X$ for which $\TC(X)$ is greater than the zero-divisors cup-length for any field of coefficients.

In this paper we give new lower bounds for $\TC(G)$ which are described in terms of the subgroup structure of $G$. These lower bounds do not, therefore, require knowledge of the cohomology algebra of $G$ or its cohomology operations.

\begin{thm}\label{main1}
Let $G$ be a discrete group, and let $A$ and $B$ be subgroups of $G$. Suppose that $gAg^{-1}\cap B=\{1\}$ for every $g\in G$.
Then $\TC(G)\ge \chd(A\times B)$.
\end{thm}

Thus $\TC(G)$ is bounded below by the cohomological dimension of the direct product $A\times B$ if no non-trivial element of $A$ is conjugate in $G$ to an element of $B$. Note that $A\times B$ is not a subgroup of $G$ in general, and so it may well happen that $\chd(A\times B)>\chd(G)$.

From another viewpoint, Theorem \ref{main1} implies that upper bounds for the topological complexity of $G$ force certain pairs of subgroups of $G$ to contain non-trivial conjugate elements.

We next note some general settings in which the assumptions of Theorem \ref{main1} are satisfied. Recall that subgroups $A$ and $B$ of $G$ are \emph{complementary} if $A\cap B=\{1\}$ and $G=AB$. Then for every $g\in G$ we can write $g^{-1}=\alpha\beta$ for some $\alpha\in A$ and $\beta\in B$, and the condition $gAg^{-1}\cap B=\{1\}$ follows easily from $A\cap B=\{1\}$. In the special case when either $A$ or $B$ is normal in $G$, then $G$ is a semi-direct product.

\begin{cor}\label{main2}
Let $G$ be a discrete group, and let $A$ and $B$ be complementary subgroups of $G$.
Then $\TC(G)\ge \chd(A\times B)$.
\end{cor}

\begin{cor}\label{semidirect}
If $G=A\rtimes B$ is a semidirect product, then $\TC(G)\ge \chd(A\times B)$.
\end{cor}

These results should be compared with results in our companion paper \cite{GLO}, in which we treat non-aspherical spaces.

The paper is organized as follows. In Section 2 we recall the necessary definitions and preliminaries concerning category, sectional category and cohomological dimension. Section 3 is a gallery of examples illustrating how our Theorem \ref{main1} may be applied. For right-angled Artin groups, we recover the lower bounds for their topological complexity obtained by Cohen and Pruidze in \cite{C-P08}. In the case of pure braid groups, the lower bounds for their topological complexity obtained by Farber and Yuzvinsky in \cite{FY} follow from our Theorem \ref{main1} using an appealing geometric argument. We also consider the Borromean rings, and recover the conclusion of \cite[Example 4.3]{G1}, thus showing that our bounds can improve on zero-divisors cup-length bounds with very little computational effort. In Section 4 we consider Higman's curious acyclic group, introduced in \cite{H51}, and show that its topological complexity is 4. Since this group is known to have trivial cohomology for all finitely generated coefficient modules, it seems unlikely that this result could be obtained by standard cohomological methods. Finally, in Section 5 we prove Theorem \ref{main1}. The proof, which uses the notion of $1$-dimensional category due to Fox \cite{Fo41}, can be read independently of Sections 3 and 4.

Several of the arguments in Section 4 use Bass--Serre theory, and are due to Y.\ de Cornulier. We warmly thank him for his input.

\section{Definitions, notations and preliminaries}

We begin by recalling the definitions of Lusternik--Schnirelmann category and sectional category. Further details can be found in the reference \cite{CLOT03}. All spaces are assumed to have the homotopy type of a CW complex, and all groups are considered discrete.

\begin{defn} The \emph{(Lusternik--Schnirelmann) category} of a space $X$, denoted $\cat(X)$, is defined to be the least integer $k$ for which $X$ admits a cover by $k+1$ open sets $U_0, \ldots , U_k$ such that each inclusion $U_i\hookrightarrow X$ is null-homotopic. If no such integer exists, we set $\cat(X)=\infty$.
\end{defn}
\begin{defn}
Let $p\co E\to B$ be a continuous map of spaces. The \emph{sectional category} of $p$, denoted $\secat(p)$, is defined to be the least integer $k$ for which $B$ admits a cover by $k+1$ open sets $U_0,\ldots , U_k$ on each of which there exists a partial homotopy section of $p$ (that is, a continuous map $s_i\co U_i\to E$ such that $p\circ s_i$ is homotopic to the inclusion $U_i\hookrightarrow B$). If no such integer exists, we set $\secat(p)=\infty$.
\end{defn}

When $p\co E\to B$ is a (Hurewicz) fibration, an open subset $U\subseteq B$ admits a partial homotopy section if and only if it admits a partial section (that is, a map $\sigma\co U\to E$ such that $p\circ \sigma$ equals the inclusion $U\hookrightarrow B$). The sectional category of fibrations was first studied systematically by \v Svarc (under the name \emph{genus}) in \cite{Sv66}, where proofs of the following results can be found.

\begin{prop} \label{surjnull}
Let $p\co E\to B$ be a fibration.
\be
\item If $B$ is path-connected and $E$ is non-empty, then $\secat(p)\le \cat(B)$.
\item If $p$ is null-homotopic, then $\secat(p)\ge \cat(B)$.
\ee
\end{prop}

\begin{prop} \label{pullbackbound}
Suppose that the diagram
\[
\xymatrix{
A \ar[r] \ar[d]_q & E \ar[d]^p \\
Y \ar[r] & B }
\]
is a homotopy pullback. Then $\secat(q)\le \secat(p)$.
\end{prop}

Let $X$ be a space. The \emph{free path fibration on $X$} is the fibration
\[
\pi_X\co X^I\to X\times X,\qquad \pi_X(\gamma) = \big( \gamma(0),\gamma(1)\big),
\]
where $X^I$ denotes the space of all paths in $X$ endowed with the compact-open topology. This has fiber $\Omega X$, the subspace of $X^I$ consisting of paths which begin and end at some fixed point $x_0\in X$.

\begin{defn}
The \emph{topological complexity} of a space $X$, denoted $\TC(X)$, is defined to be $\secat(\pi_X)$, the sectional category of the free path fibration on $X$.
\end{defn}

Topological complexity is therefore an important special case of sectional category. As well as having applications to the motion planning problem in robotics (as explained in the original articles \cite{Far03,Far04,Far06}) it also turns out to give an equivalent formulation of the immersion problem for real projective spaces \cite{FTY}. It is a close relative of the LS-category, although the two notions are independent. Note that when $X$ is path-connected, simple arguments using Propositions \ref{surjnull} (1) and \ref{pullbackbound} yield the inequalities
\[
\cat(X) \le \TC(X)\le \cat(X\times X).
\]
Either inequality can be an equality, as illustrated by the orientable surfaces (see \cite[Theorem 9]{Far03}). Note also that when $X$ is a CW complex, combining this upper bound with the standard dimensional upper bound for category yields
\[
\TC(X)\le 2\cdot\dim(X).
\]

In this paper we consider the problem of calculating the topological complexity of aspherical spaces. Given a group $G$, one can construct an aspherical based CW complex $K(G,1)$, whose homotopy type is an invariant of the isomorphism class of $G$. One can therefore define $\cat(G):=\cat(K(G,1))$ and $\TC(G):=\TC(K(G,1))$, and then ask for results which describe these invariants in purely algebraic terms. In the case of category, such results were obtained by Eilenberg--Ganea, Stallings and Swan, in terms of cohomological dimension.

\begin{defn}
The \emph{cohomological dimension} of a group $G$, denoted $\chd(G)$, is defined to be the least integer $k$ such that $H^i(G;M)=0$ for all $i>k$ and all coefficient $\Z[G]$-modules $M$. If no such integer exists, we set $\chd(G)=\infty$.
\end{defn}

\begin{thm}[Eilenberg--Ganea \cite{EG}, Stallings \cite{Sta}, Swan \cite{Swa}]
For any group $G$, we have $\cat(G)=\chd(G)$.
\end{thm}

This result identifies $\cat(G)$ with a well-known, algebraically defined invariant of $G$, namely its cohomological dimension (which agrees with the projective dimension of $\Z$ as a $\Z[G]$-module, see Brown \cite[Chapter VIII]{Bro}). One might hope for a similar result for $\TC(G)$. However, no such result is known, even in conjectural form. Note that if $G$ has torsion (and in particular if $G$ is finite) then $\chd(G)=\cat(G)=\infty$, and so $\TC(G)=\infty$ also. Therefore, in this paper we will consider only torsion-free infinite groups of finite cohomological dimension.

We conclude this section with two lemmas on cohomological dimension of direct products, which will be useful in applying our Theorems \ref{main1} and \ref{main2} in specific cases. It is known that the cohomological dimension behaves sub-additively, in that $\chd(A\times B)\le \chd(A) + \chd(B)$ for any groups $A$ and $B$. It is also known that the inequality may be strict (for instance, it follows from \cite{Var} that $\chd(\Q\times \Q)=3<4=\chd(\Q)+\chd(\Q)$). This strictness cannot occur, however, if the groups satisfy certain duality hypotheses.

Recall from \cite{BiEck} that a group $A$ is called a \emph{duality group of dimension $k$} if there exists some $\Z[A]$-module $C$ and an element $e\in H_k(A;C)$ such that cap product with this element gives an isomorphism
\[
-\cap e \co H^i(A;M)\cong H_{k-i}(A;M\otimes C)
\]
for all $i$ and all $\Z[A]$-modules $M$. Note that in this case, $\chd(A)=k$. If $C$ can be chosen to have underlying abelian group $\Z$, then $A$ is called a \emph{Poincar\' e duality group of dimension $k$}, or $PD_k$ group for short. If, in addition, $C$ can be chosen to be $\Z$ with the trivial module structure, then $A$ is called an \emph{orientable $PD_k$ group}.

For instance, the fundamental group of a closed (orientable) aspherical $k$-manifold is an (orientable) $PD_k$ group. There are many examples of duality groups which are not Poincar\' e duality groups, such as knot groups and Baumslag--Solitar groups \cite{BiEck}.

\begin{lem}[{\cite[Theorem 3.5]{BiEck}}]\label{dualitygroups}
Let
\[
\xymatrix{
1 \ar[r] & B\ar[r] & \Gamma \ar[r] & A \ar[r] & 1
}
\]
be an extension of groups in which $A$ and $B$ are duality groups, of dimensions $k$ and $\ell$ respectively. Then $\Gamma$ is a duality group of dimension $k+\ell$. In particular, if $A$ and $B$ are duality groups then $\chd(A\times B)=\chd(A)+\chd(B)$.
\end{lem}

\begin{lem}\label{dimAtimesZ}
If $A$ is an orientable Poincar\' e duality group, and $B$ is any group, then $\chd(A\times B) = \chd(A) + \chd(B)$.
\end{lem}
\begin{proof}
Let $M$ be a $\Z[B]$-module such that $H^\ell(B;M)\neq 0$, where $\ell=\chd(B)$. We will consider the Lyndon--Hochschild--Serre spectral sequence of the trivial extension
\[
\xymatrix{
1 \ar[r] & B\ar[r] & A\times B \ar[r] & A \ar[r] & 1
}
\]
which has $E_2^{p,q} = H^p(A ; H^q(B;M))$ and converges to $H^\ast(A\times B ; M)$. Here we regard $M$ as a $\Z[A\times B]$-module via the homomorphism $\Z[A\times B]\to \Z[B]$.

Letting $k=\chd(A)$, we see that
$$
E_2^{k,\ell} = H^k(A; H^\ell(B;M)) \cong H_0(A;H^\ell(B;M)) \cong H^\ell(B;M),
$$
 because $A$ acts trivially on $H^\ell(B;M)$. There can be no nonzero differentials associated with this group, and so our assumptions on the dimensions of $A$ and $B$ imply that $H^{k+\ell}(A\times B;M)\cong E_{\infty}^{k,\ell} \cong E_2^{k,\ell}\neq 0$. Hence $\chd(A\times B)\ge k+\ell$, as required.
\end{proof}

\section{Examples}

In this section we present three examples illustrating how our Theorem \ref{main1} may be applied for various groups. In the first two examples (right-angled Artin groups and pure braid groups) we recover sharp lower bounds, originally obtained using zero-divisors cup-length. In the final example we consider the link complement of the Borromean rings. Here our lower bound improves on the zero-divisors cup-length, and recovers the bound of \cite[Example 4.3]{G1} which was obtained using sectional category weight and Massey products.

\subsection{Right-angled Artin groups}

Right-angle Artin groups form a rich class of infinite groups, including free groups of finite rank and finitely-generated abelian groups. To any finite simple graph $\Gamma$, one can associate a right-angled Artin group $G_\Gamma$ as follows. The group $G_\Gamma$ has a finite presentation, with one generator $x_i$ for each vertex $v_i\in \mathcal{V}(\Gamma)$, and one relation $x_ix_j=x_jx_i$ for each edge $\{v_i,v_j\}\in\mathcal{E}(\Gamma)$. Thus $G_\Gamma$ has a finite presentation in which all relators are commutators.

The topological complexity of right-angle Artin groups was computed by Cohen and Pruidze in  \cite{C-P08}.

\begin{thm}\cite[Theorem 4.1]{C-P08}
For any finite simplicial graph $\Gamma$, the topological complexity of the associated right-angled Artin group is given by $\TC(G_\Gamma)=z(\Gamma)$, where
 \[
 z(\Gamma) = \max_{K_1,K_2}\left| \mathcal{V}(K_1)\cup \mathcal{V}(K_2)\right|
 \]
 is the maximum number of vertices spanned by precisely two cliques in $\Gamma$.
\end{thm}
Recall that a clique $K$ in $\Gamma$ is a complete subgraph. By an abuse of notation we conflate a clique with its set of vertices. So we may write, for instance, $K=\{ v_{i_1},\ldots , v_{i_m}\}$ and $|K|=m$. Note that in the above Theorem, the two cliques $K_1$ and $K_2$ attaining $z(\Gamma)$ may, without loss of generality, be taken to be disjoint.

Here we observe that the lower bound of $\TC(G_\Gamma)\ge z(\Gamma)$, which was obtained in \cite{C-P08} using zero-divisors calculations, follows easily from our Theorem \ref{main1}.

\begin{prop}
Let $K_1$ and $K_2$ be disjoint cliques in $\Gamma$. Then $\TC(G_\Gamma)\ge |K_1|+|K_2|$.
\end{prop}

\begin{proof}
Suppose that $K_1=\{ v_{i_1},\ldots , v_{i_m}\}$ and $K_2 = \{ v_{j_1},\ldots , v_{j_n} \}$ are disjoint cliques, and let $A=\langle x_{i_1},\ldots , x_{i_m}\rangle$ and $B=\langle x_{j_1},\ldots , x_{j_n} \rangle$ be the free abelian subgroups spanned by the corresponding generators of $G_\Gamma$. The inclusion $B\hookrightarrow G_\Gamma$ admits a retraction $\varphi\co G_\Gamma\to B$ which sends $x_j$ to $0$ if $v_j\notin K_2$, and is the identity on $B$. Note that $\varphi(gAg^{-1})=\{1\}$ for every $g\in G$. It follows that $gAg^{-1}\cap B=\{1\}$. Thus Theorem \ref{main1} gives $$\TC(G_\Gamma)\ge\chd(A\times B) = \chd(\Z^m\times\Z^n) = m+n,$$ as required.
\end{proof}

\subsection{Pure braid groups}

Let $\PB_n$ denote the pure braid group on $n\ge2$ strands. We regard a braid in the standard way as an isotopy class of $n$ non-intersecting strands in $\R^3$, monotonic in the $z$-coordinate and connecting $n$ distinct points in the plane $z=1$ with the corresponding points in the plane $z=0$. The group operation is given by concatenating braids and re-scaling. We may depict a braid\footnote{Courtesy of Andrew Stacey's \texttt{braids} package (\texttt{http://www.ctan.org/pkg/braids}).} by a diagram such as in Figure \ref{1}.

\begin{figure}%[h!]
\begin{tikzpicture}[scale=0.7]
\braid[thick] a_1 a_2a_2 a_3a_3a_1 ;
\end{tikzpicture}
\caption{A pure braid. \label{1}}
\end{figure}
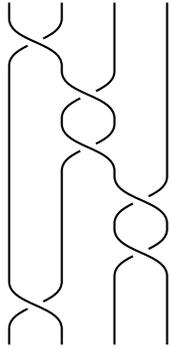

Recall that we may take as a $K(\PB_n,1)$ space the complement of the standard braid arrangement in $\C$, otherwise known as the configuration space
\[
F(\C,n) = \{ (z_1,\ldots , z_n) \in \C^n \mid i\neq j \implies z_i\neq z_j \}.
\]
The topological complexity of this space was computed in \cite{FY} to be $2n-3$. The lower bound can be obtained by considering the zero-divisors cup-length in cohomology with rational coefficients. The upper bound follows from the existence of a homeomorphism $F(\C,n) \approx \C^* \times M$, where $M$ is a complex of homotopy dimension $n-2$, together with the product formula \cite[Theorem 11]{Far03} and the standard dimensional upper bound $\TC(M)\le 2\,\dim(M)$. Note that $\chd(\PB_n)=n-1$ (see \cite{Arn}, for example).

Here we will recover the lower bound $\TC(\PB_n)\ge 2n-3$, using Theorem \ref{main1} instead of a zero-divisors calculation.

\begin{prop}
We have $\TC(\PB_n) \ge 2n-3$ for all $n\ge 2$.
\end{prop}

\begin{proof}
We will identify subgroups $A$ and $B$ of $\PB_n$ such that $gAg^{-1}\cap B=\{1\}$ for all $g\in\PB_n$, and $\chd(A\times B)=2n-3$.

For each $1\le j\le n-1$, define a braid $\alpha_j\in\PB_n$ as follows. Geometrically, the braid $\alpha_j$ runs the $j$-th strand in front of the last $n-j$ strands, then loops back and passes behind the last $n-j$ strands to its original position. This is depicted in Figure \ref{3a}. It is easy to see geometrically that these elements commute pairwise, as illustrated by Figure \ref{3b}.

\begin{figure}%[h]
\begin{tikzpicture}[scale=0.7]
\braid[number of strands=4,thick,style strands={1}{red}] a_1a_2a_3a_3a_2a_1;
\end{tikzpicture}
\caption{The braid $\alpha_1\in\PB_4$. \label{3a}}
\end{figure}
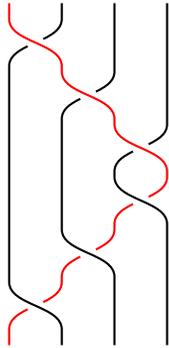

\begin{figure}%[h]
\begin{tikzpicture}[scale=0.7]
\braid[number of strands=4,thick,style strands={3}{green}, style strands={1}{red}](b2b4) at (0,0) a_1a_2a_3a_3a_2a_1a_3a_3;
\node at (4, -3.5){$=$};
\braid[number of strands=4,thick,style strands={3}{green}, style strands={1}{red}](b4b2) at (5,0) a_3a_3a_1a_2a_3a_3a_2a_1;
\end{tikzpicture}
\caption{The relation $\alpha_1\alpha_3=\alpha_3\alpha_1\in\PB_4$. \label{3b}}
\end{figure}
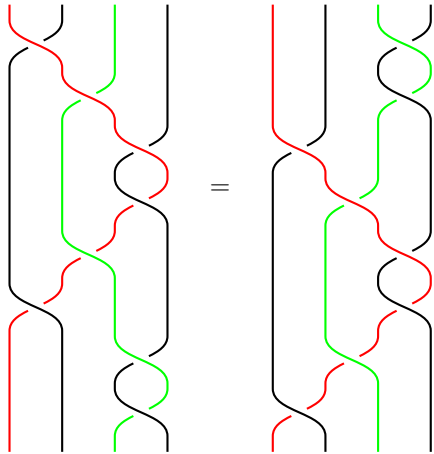

Since the pure braid groups are torsion-free, the $\alpha_j$ generate a free abelian subgroup $A=\langle \alpha_1,\ldots ,\alpha_{n-1}\rangle\le \PB_n$ of rank $n-1$.

Now recall that there is a canonical inclusion $\PB_{n-1}\hookrightarrow\PB_n$, given by introducing an $n$-th non-interacting strand to the right of the other strands. We let $B\le \PB_n$ denote the image of this inclusion, and note that $\chd(B)=\chd(\PB_{n-1})=n-2$.

We claim that $gAg^{-1}\cap B = \{1\}$ for every $g\in \PB_n$. To see this, recall that any pure braid on $n$ strands may be \emph{closed} to form an ordered link of $n$ components. If two braids $\alpha$ and $\beta$ are conjugate in $\PB_n$, then their closures are isotopic as ordered links. Suppose $\alpha\in A$ is conjugate to $\beta\in B$. Then the closure of $\alpha$ is an ordered link in which each component $L_1,\ldots ,L_{n-1}$ has trivial linking number with the final component $L_n$. On the other hand, it is easy to see that in terms of the given basis for $A$ we have $\alpha=\alpha_1^{k_1}\cdots \alpha_{n-1}^{k_{n-1}}$, where $k_i\in\Z$ is the linking number of $L_i$ with $L_n$. Therefore $\alpha=1$, and the conjugates of $A$ intersect $B$ trivially, as claimed.

Finally we apply Theorem $\ref{main1}$, to obtain
\[
\TC(\PB_n)\ge \chd(A\times B) = \chd(\Z^{n-1}\times \PB_{n-1}) = (n-1) + (n-2) = 2n-3. \qedhere
\]
\end{proof}

\begin{rem}
The above approach was suggested by viewing the pure braid groups as iterated semi-direct products of free groups, and was originally proved using the explicit presentation of $\PB_n$ given by Artin in \cite{Art}. We have found this geometric approach, discovered later, to be more appealing.
\end{rem}

\subsection{The Borromean rings}

Let $X$ be the link complement of the Borromean rings in $S^3$ (see Figure \ref{4}). Then $X$ is a compact $3$-manifold with boundary, which is aspherical by virtue of the fact that it admits a complete hyperbolic metric of finite volume \cite{Thu}. The fundamental group $G=\pi_1(X)$ admits a finite presentation
\[
G = \langle a,b,c \mid [a,[b^{-1},c]],\, [b,[c^{-1},a]] \rangle,
\]
obtained using the Wirtinger method. Since $X$ deformation retracts onto a $2$-dimensional complex, we have $\cat(G)=\cat(X)\le 2$. On the other hand, $G$ is not free, and so $\cat(G)=\chd(G)=2$. We therefore have $2\le \TC(G)\le 4$. The fact that the rings are pairwise unlinked implies that cup products in $\widetilde{H}^\ast(X;k)$ are all trivial, for any field $k$. The zero-divisors cup-length is therefore less than or equal to $2$ (and in fact is equal to $2$, as whenever $u$ and $v$ are linearly independent cohomology classes in $\widetilde{H}^\ast(X;k)$ the product of zero-divisors $(1\times u-u\times 1)(1\times v-v\times 1)$ is nonzero). We will use our Theorem \ref{main1} to show that $\TC(G)\ge 3$. Our aim, therefore, is to find two subgroups $A$ and $B$ of $G$ whose conjugates intersect trivially and such that $\chd(A\times B)=3$.

\begin{figure}%[h!]
\begin{tikzpicture}[scale =.5]
\fill [white] (-5,-4.2)rectangle (5,5.2);
\draw [blue, thick] (1.732,-1) circle (3);

\begin{scope}
\clip(0,1.45) circle (.5);
\draw [white, ultra thick] (1.732,-1) circle (3);
\end{scope}
\begin{scope}
\clip(0,-3.45) circle (.5);
\draw [white, ultra thick] (1.732,-1) circle (3);
\end{scope}

\draw [green, thick] (-1.732,-1) circle (3);

\begin{scope}
\clip(-2.987,1.725) circle (.5);
\draw [white, ultra thick] (-1.732,-1) circle (3);
\end{scope}
\begin{scope}
\clip(1.255,-0.725) circle (.5);
\draw [white, ultra thick] (-1.732,-1) circle (3);
\end{scope}

\draw [red, thick] (0,2) circle (3);

\begin{scope}
\clip(-1.255,-0.725) circle (.5);
\draw [white, ultra thick] (0,2) circle (3);
\end{scope}
\begin{scope}
\clip  (2.987,1.725) circle (.5);
\draw [white, ultra thick] (0,2) circle (3);
\end{scope}
\begin{scope}
\clip(-1.255,-0.725) circle (.5);
\draw [blue, thick] (1.732,-1) circle (3);
\end{scope}
\begin{scope}
\clip  (2.987,1.725) circle (.5);
\draw [blue, thick] (1.732,-1) circle (3);
\end{scope}

\end{tikzpicture}
\caption{The Borromean rings. \label{4}}
\end{figure}
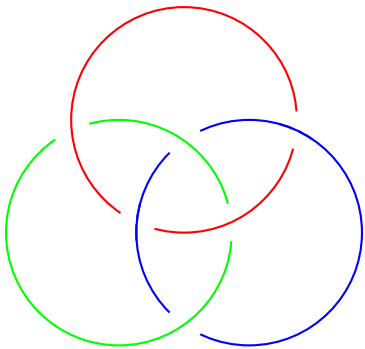

In order to do so, we decompose $G$ as a semidirect product, as follows. Recall that removing any one component of the Borromean rings results in the unlink of two components. Suppose we remove the component corresponding to the generator $c$. Then the induced map on fundamental groups of link complements is a homomorphism
\[
p\co G\to F(\alpha,\beta),\qquad a\mapsto \alpha,\quad b\mapsto \beta, \quad c\mapsto 1,
\]
where $F(\alpha,\beta)$ is a free group on two generators $\alpha$ and $\beta$. Since $p$ is clearly onto, and any surjective homomorphism to a free group splits, there results a split extension
\[
\xymatrix{
 1 \ar[r] &  K\ar[r] & G\ar[r]^-p &F(\alpha,\beta) \ar[r] & 1.
}
\]

\begin{lem}
Let $A=\langle a\rangle$ be the infinite cyclic subgroup of $G$ generated by $a$, and let $B=p^{-1}\langle\beta\rangle$, where $\langle \beta \rangle\le F(\alpha,\beta)$ denotes the infinite cyclic subgroup generated by $\beta$. Then for every $g\in G$ we have $gAg^{-1}\cap B=\{1\}$.
\end{lem}
\begin{proof}
Suppose $x\in gAg^{-1}\cap B \le G$. Then $p(x)\in p(g)Ap(g)^{-1}\cap \langle\beta\rangle = \{1\}\le F(\alpha,\beta)$. Therefore $x=ga^{n}g^{-1}\in K$, for some $n\in \Z$. Since $K$ is normal, it follows that $a^n\in K$, and so $n=0$ and $x=1$.
\end{proof}

\begin{prop}\label{rings}
We have $\TC(G)\ge 3$, where $G$ denotes the fundamental group of the Borromean rings link complement.
\end{prop}
\begin{proof}
By the previous Lemma combined with Theorem \ref{main1}, we have $\TC(G)\ge \chd(A\times B) = 1 + \chd(B)$. All that remains is to observe that $B$ is not a free subgroup of $G$, and so $\chd(B)=2$. To see this, note that $b$ and $[c^{-1},a]$ are both in $B=p^{-1}\langle\beta\rangle$, and that $[b,[c^{-1},a]]=1$ in $B$. However, $b^n\neq [c^{-1},a]$ for all $n\in\Z$ (since $p(b^n)=\beta^n \neq 1 = p([c^{-1},a])$ if $n\neq 0$, and $b^{0}=1 \neq [c^{-1},a]$). This non-trivial relation shows that $B$ is not a free group.
\end{proof}

\begin{rem}
Proposition \ref{rings} recovers Example 4.3 in the paper \cite{G1}, where Massey product calculations and sectional category weight are used to achieve the same lower bound (note that in that paper the non-normalized $\TC$ is used). The proof presented here is somewhat simpler and more illuminating. The method of proof also generalizes immediately to other aspherical Brunnian link complements (such as the complement of Whitehead's link).
\end{rem}

\section{Higman's group}

In our previous examples, we used Theorem \ref{main1} to arrive at lower bounds for $\TC(G)$ which can also be obtained using zero-divisors cup-length, or sectional category weight. In this section, we present a calculation of the topological complexity of a group which, to the best of our knowledge, could not be obtained using these standard techniques.

In his paper \cite{H51}, G.\ Higman gave an example of a $4$-generator, $4$-relator group with several remarkable properties. Here we recall the construction of Higman's group as an iterated amalgam, as well as those properties which are relevant for our purposes. More information may be found in \cite{H51} or \cite[Section 4]{DV73}.

Firstly, for symbols $x$ and $y$ form the group $H_{xy}$ with presentation
\[
\langle x,y \mid xyx^{-1}y^{-2} \rangle.
\]
This group is isomorphic to the Baumslag--Solitar group $B(1,2)$, and hence is a duality group of dimension $2$.

The infinite cyclic group $F(y)$ injects into both $H_{xy}$ and $H_{yz}$, and so we may form the amalgam $H_{xyz}:=H_{xy}\ast_{F(y)} H_{yz}$. Likewise we may form $H_{zwx}$ as the amalgam of $H_{zw}$ and $H_{wx}$ over $F(w)$. The free group $F(x,z)$ injects into both $H_{xyz}$ and $H_{zwx}$, and Higman's group is defined to be the amalgam $\mathcal{H}:=H_{xyz}\ast_{F(x,z)} H_{zwx}$. It has presentation
\[
P\co \langle x , y, z, w \mid xyx^{-1}y^{-2},\, yzy^{-1}z^{-2},\, zwz^{-1}w^{-2},\, wxw^{-1}x^{-2} \rangle.
\]
Note the symmetry in this presentation. Indeed, the construction of $\mathcal{H}$ as an iterated amalgam is non-unique, and below we will make use of a second amalgam decomposition $\mathcal{H}=H_{yzw}\ast_{F(y,w)} H_{wxy}$.

The group $\mathcal{H}$ is acyclic (it has the same integer homology as a trivial group), and so $\tilde{H}^\ast(\mathcal{H};k)=0$ for every abelian group $k$. Moreover, since $\mathcal{H}$ has no non-trivial finite quotients \cite{H51}, it has no non-trivial finite dimensional representations over any field. It then follows that if $M$ is any coefficient $\Z[\mathcal{H}]$-module which is finitely generated as an abelian group, then $\tilde{H}^\ast(\mathcal{H};M)=0$. Thus the group $\mathcal{H}$ is difficult to distinguish from a trivial group using cohomological invariants.

On the other hand, since $\mathcal{H}$ is not a free group we have $\chd(\mathcal{H})\ge 2$. The $2$-dimensional complex associated to the presentation $P$ is shown to be aspherical in \cite{DV73}, and it follows that $\cat(\mathcal{H})=\chd(\mathcal{H})=2$. Thus the topological complexity of Higman's group satisfies $2\le \TC(\mathcal{H})\le 4$. Note that the zero-divisors cup length over any field is zero. In this section we will prove the following result.

\begin{thm}\label{Higman}
We have $\TC(\mathcal{H})=4$.
\end{thm}

  We will employ results on the structure of amalgams and their Bass--Serre trees, for which Serre's monograph \cite{Ser} is the standard reference. The fundamental result \cite[Theorem 7]{Ser} gives a correspondence between $G$-trees $X$ with quotient graph $X/G$ a segment, and amalgam structures on the group $G$. More precisely, let $G=A\ast_C B$ be an amalgam. Then there exists a $G$-tree $X$, unique up to isomorphism, with fundamental domain a segment $T\subseteq X$ consisting of vertices $v, w\in\mathcal{V}(X)$ and an edge $e=(v,w)\in\mathcal{E}(X)$, such that the stabilizer subgroups are given by $G_v=A$, $G_w=B$ and $G_e=C$, and the natural inclusions $G_e\hookrightarrow G_v$ and $G_e\hookrightarrow G_w$ agree with the defining monomorphisms $C\hookrightarrow A$  and $C\hookrightarrow B$. The tree $X$ together with its action is called the \emph{Bass--Serre tree} of the amalgam $G=A\ast_C B$.

   We will also make use of the elliptic--hyperbolic dichotomy for automorphisms of trees. Let $G$ be a group acting on a tree $X$ (by automorphisms, of course). Recall that a tree is a graph characterized by the property that, for any two vertices $v,w\in \mathcal{V}(X)$, there is a unique geodesic path $\gamma$ in $X$ from $v$ to $w$. The distance from $v$ to $w$ in $X$, denoted $d_X(v,w)$, is defined to be the number of edges in $\gamma$. Now for each $g\in G$ we put
   \[
   \ell_X(g) = \min \{ d_X(v,gv)\mid v\in \mathcal{V}(X)\}.
   \]
   This defines a function $\ell_X\co G\to \mathbb{Z}$, called the \emph{hyperbolic length function} for the action of $G$ on $X$. For each element $g\in G$ there are two possibilities:
   \begin{itemize}
   \item $\ell_X(g)=0$, and $g$ fixes a vertex $v\in \mathcal{V}(X)$. In this case we say that $g$ is \emph{elliptic}, or that \emph{$g$ acts elliptically on $X$}.
   \item $\ell_X(g)>0$, in which case we say that $g$ is \emph{hyperbolic}, or that \emph{$g$ acts hyperbolically on $X$}.
   \end{itemize}
   There are two trivial observations worth making at this point. One is that the hyperbolic length function descends to a function on the conjugacy classes of $G$. In other words, for all $g, h\in G$ we have that $\ell_X(hgh^{-1})=\ell_X(g)$. The second is that if $g\in G$ acts elliptically on $X$, then for any vertex $w\in \mathcal{V}(X)$ the sequence $\{d_X(g^mw,w)\}_m$ is bounded. To see this, consider the geodesic from $w$ to some vertex $w_0$ fixed by $g$. Its image under $g^m$ is the geodesic from $g^mw$ to $w_0$. Composing these two geodesics gives a path from $g^mw$ to $w$, of length $2\cdot d_X(w,w_0)$.

  % The following lemma and its proof were communicated to us by Yves de Cornulier.

   \begin{lem} \label{hyperbolic}
   Suppose that the free product $G=A\ast B$ acts on a tree $X$ containing an edge $(v,w)\in\mathcal{E}(X)$, such that $A$ fixes $v$ and $B$ fixes $w$, while no element of $A-1$ fixes $w$ and no element of $B-1$ fixes $v$. Then any element $g\in G$ not conjugate to an element of the union $A\cup B$ is hyperbolic.
   \end{lem}
   \begin{proof} Suppose that $g\in G$ is not conjugate to an element of $A\cup B$. Taking conjugates if necessary (and without loss of generality), we may suppose that $g=a_1b_1a_2b_2\cdots a_kb_k$, where $k\geq 1$ and each $a_i\in A-1$ and $b_i\in B-1$. We will show that for any $g$ of this form we have
   \begin{equation}\label{distance}
   d_X(gw,v)=2k-1\qquad\mbox{and}\qquad d_X(gw,w)=2k.
    \end{equation}
    In particular, since $g^m$ is of the same form for $m\in\mathbb{N}$, the distances $d_X(g^mw,w)=2mk$ grow linearly with $m$. This implies that $g$ is hyperbolic, by the observation made just prior to the statement of the lemma.

      To establish (\ref{distance}) we proceed by induction on $k$. If $k=1$ and $g=a_1b_1$, then $gw=a_1w$ and $v=a_1v$ are connected by the edge $a_1(w,v)=(a_1w,v)$. Following this by the edge $(v,w)$ gives a path of length $2$ from $gw$ to $w$. Note that $X$ cannot contain the edge $(gw,w)$, as this would result in a triangular circuit. This completes the base case.

      Now suppose that $k>1$ and write $g=a_1b_1h$. By induction, the geodesic from $hw$ to $v$ has length $2k-3$, and the geodesic from $hw$ to $w$ has length $2k-2$ and finishes with $(v,w)$. Since $b_1$ fixes $w$ and does not fix $v$, the geodesic from $b_1hw$ to $w$ has length $2k-2$ and does not finish with $(v,w)$. Therefore the geodesic from $b_1hw$ to $v$ has length $2k-1$ and finishes with $(w,v)$. Now, since $a_1$ fixes $v$ and does not fix $w$, the geodesic from $gw=a_1b_1hw$ to $v$ has length $2k-1$ and does not finish with $(w,v)$. Therefore the geodesic from $gw$ to $w$ has length $2k$ and finishes with $(v,w)$. This completes the induction, and the lemma is proved.
   \end{proof}

   Finally before embarking on the proof of Theorem \ref{Higman}, we state the following two lemmas concerning conjugacy in amalgams, which can be proved directly using the Structure Theorem for amalgams (see \cite[Theorem 2]{Ser}).

\begin{lem}\label{amalgam1}
In an amalgam $G=A\ast_C B$, if an element of $A$ is conjugate in $G$ to an element of $B$, then it is conjugate in $G$ to an element of $C$.
\end{lem}

\begin{lem}\label{amalgam2}
In an amalgam $G=A\ast_C B$, if an element of $A$ is conjugate in $G$ to an element of $C$, then it is conjugate in $A$ to an element of $C$.
\end{lem}

\begin{proof}[Proof of Theorem \ref{Higman}]
Let $H_{xy}=\langle x,y \rangle\le \mathcal{H}$ be the subgroup generated by $x$ and $y$, and let $H_{zw}=\langle z,w \rangle\le \mathcal{H}$ be the subgroup generated by $z$ and $w$. Both $H_{xy}$ and $H_{zw}$ are isomorphic to the Baumslag--Solitar group $B(1,2)$, hence are duality groups of dimension $2$. By Lemma \ref{dualitygroups} their product $H_{xy}\times H_{zw}$ is a duality group of dimension $4$, and so $\chd(H_{xy}\times H_{zw})=4$. The proof of Theorem \ref{Higman} will be complete if we can show that the hypothesis of Theorem \ref{main1} applies to these subgroups. Therefore we must show that $gH_{xy}g^{-1}\cap H_{zw}=\{1\}$ for all $g\in \mathcal{H}$.

 Suppose that $a\in H_{xy}$ is conjugate in $\mathcal{H}$ to $b\in H_{zw}$; we aim to show that $a=1$. Using the first description of $\mathcal{H}$ as an amalgam $H_{xyz}\ast_{F(x,z)} H_{zwx}$ and Lemma \ref{amalgam1}, we find that $a$ is conjugate in $\mathcal{H}$ to some element of $F(x,z)$. Using the second description of $\mathcal{H}$ as an amalgam $H_{yzw}\ast_{F(y,w)} H_{wxy}$ and Lemma \ref{amalgam1} again, we find that $b$ is conjugate in $\mathcal{H}$ to some element of $F(y,w)$. We are thus reduced to showing that an element $\alpha$ of $F(x,z)$ conjugate in $\mathcal{H}$ to an element $\beta$ of $F(y,w)$ must be trivial.

Consider the first description of $\mathcal{H}$ as an amalgam $H_{xyz}\ast_{F(x,z)} H_{zwx}$, and its corresponding Bass--Serre tree $X$. The element $\alpha\in F(x,z)$ fixes an edge of the tree, and therefore acts elliptically. The element $\beta\in F(y,w)$ therefore also acts elliptically on $X$. Restricting to an action of $F(y,w)=F(y) \ast F(w)$ on $X$ and applying Lemma \ref{hyperbolic}, we find that $\beta$ must be conjugate in $F(y,w)$ to an element of $F(y)\cup F(w)$. We conclude that $\beta$ is conjugate in $\mathcal{H}$ to some element $\beta' \in F(y)\cup F(w)$.

The same argument using the description of $\mathcal{H}$ as an amalgam $H_{yzw}\ast_{F(y,w)} H_{wxy}$ shows that $\alpha\in F(x,z)$ is conjugate in $\mathcal{H}$ to some $\alpha'\in F(x)\cup F(z)$.

We are thus reduced to showing that an element $\alpha'\in F(x)\cup F(z)$ conjugate in $\mathcal{H}$ to an element $\beta'\in F(y)\cup F(w)$ must be trivial.

     Consider first the case that $\alpha' =x^m\in F(x)$ and $\beta'=y^n\in F(y)$. Using the first description of $\mathcal{H}$ as an amalgam $H_{xyz}\ast_{F(x,z)} H_{zwx}$ and Lemma \ref{amalgam2}, we find that $x^m$ is conjugate to $y^n$ in the group $H_{xyz}$ with presentation $\langle x,y,z \mid [x,y]=y,\, [y,z]=z \rangle$. Since non-trivial powers of $x$ survive in the abelianization while powers of $y$ are trivial, we find that $m=0$ and so $\alpha'=1$.

     Secondly, consider the case that $\alpha'=x^m\in F(x)$ and $\beta'=w^n\in F(w)$. Using the second description of $\mathcal{H}$ as an amalgam $H_{yzw}\ast_{F(y,w)} H_{wxy}$ and Lemma \ref{amalgam2}, we find that $x^m$ is conjugate to $w^n$ in $H_{wxy}$. This time powers of $w$ survive in the abelianization while powers of $x$ are trivial, and again we find that $\alpha'=1$.

     The remaining two cases are dealt with similarly. \end{proof}

\begin{rem} The arguments of this section, due to Y.\ de Cornulier \cite{Cor}, indicate that the topological complexity of groups is a much deeper property than category. They also hint at the possibility of proving a more general result about the topological complexity of free products with amalgamation. Note that for free products we have $\TC(G\ast H)\ge \chd(G\times H)$, as follows easily from our results, or from results of Dranishnikov on the topological complexity of wedge sums \cite{Dra}.
\end{rem}

\section{Proof of Theorem \ref{main1}}

In order to prove Theorem \ref{main1}, we find it convenient to use the \emph{$1$-dimensional category} of R.\ H.\ Fox \cite{Fo41} (see also \cite{OS12} and \cite{Sv66}). Let $X$ be a topological space.

\begin{defn} An open subset $U\subseteq X$ is \emph{$1$-categorical} if every composition $L\to U\hookrightarrow X$, where $L$ is a complex with $\dim(L)\le 1$, is null-homotopic. The \emph{$1$-dimensional category} of $X$, denoted $\cat_1(X)$, is the least integer $k$ for which $X$ admits a cover by $k+1$ open sets $U_0,\ldots , U_k$, each of which is $1$-categorical.
\end{defn}

We recall the following facts about $\cat_1$, proofs of which can be found in \cite[Proposition 44]{Sv66} and \cite[Corollary 2.2]{OS12}.

\begin{prop}\label{cat1props}
Let $X$ be a connected complex.
\be
\item We have $\cat_1(X)=\secat(\widetilde X \to X)$, the sectional category of the universal cover.
\item If $X$ is a $K(G,1)$, then $\cat_1(X)=\cat(X)=\chd(G)$.
\ee
\end{prop}

We also need the following lemma.

\begin{lem}\label{circles}
Let $X$ be a connected complex. An open set $U\subseteq X$ is $1$-categorical if and only if every composition $S^1\to U\hookrightarrow X$ is null-homotopic.
\end{lem}

\begin{proof}
The only if statement is trivial. Suppose that every composition $S^1\to U\hookrightarrow X$ is null-homotopic. It follows that the induced map $\pi_1(U,x)\to \pi_1(X,x)$ is trivial, for every choice of basepoint $x\in U$ (since based maps to $X$ which are null-homotopic are also based null-homotopic). So we may apply the lifting criterion \cite[Proposition 1.33]{Hat} to each path component of $U$ to conclude that $U\hookrightarrow X$ lifts through $\widetilde X\to X$. This means that every composition $L\to U\hookrightarrow X$ with $\dim(L)\le 1$ lifts through $\widetilde X\to X$ and, by
cellular approximation, is therefore null-homotopic. Hence $U$ is $1$-categorical.
\end{proof}

\begin{proof}[Proof of Theorem \ref{main1}] Let $G$ be a group with subgroups $A$ and $B$. Let $X$, $Y_A$ and $Y_B$ denote $K(\pi,1)$ complexes for $\pi=G$, $A$ and $B$ respectively. The inclusion monomorphisms induce pointed maps $Y_A\to X$ and $Y_B\to X$. There results a pullback diagram
\begin{equation}\label{pullback}
\xymatrix{
 E \ar[r] \ar[d]_q & X^I \ar[d]^{\pi_X} \\
  Y_A\times Y_B \ar[r] & X\times X \\
}
\end{equation}
from which it follows by Proposition \ref{pullbackbound} that $\secat(q)\le \secat(\pi_X) = \TC(X)$. We will show that, under the hypotheses of Theorem \ref{main1}, we have $\chd(A\times B) = \cat_1(Y_A\times Y_B) \le \secat(q)$.

          We therefore assume that $gAg^{-1}\cap B=\{1\}$ for every $g\in G$. We will show that any open set $U\subseteq Y_A\times Y_B$ which admits a partial section for $q$ must be $1$-categorical, and hence $\cat_1(Y_A\times Y_B)\le \secat(q)$. By Lemma \ref{circles}, it suffices to show that for any map $S^1\to U$, the composition $\phi\co S^1\to U\to Y_A\times Y_B$ is null-homotopic.

Applying the functor $[S^1,-]$ given by unbased homotopy classes of loops, and recalling that $\pi_X\co X^I\to X\times X$ is homotopically equivalent to the diagonal $\Delta\co X\to X\times X$, we arrive at a commutative diagram of pointed sets
\begin{equation}\label{homotopy sets}
\xymatrix{
[S^1,E] \ar[r] \ar[d]_{q_\#} & [S^1,X^I] \ar[r]^{\simeq} \ar[d]_{(\pi_X)_\#}& [S^1,X] \ar[dl]^{\Delta_\#} \\
[S^1,Y_A \times Y_B] \ar[r] & [S^1,X\times X]  &  \\
}
\end{equation}
Recall that for any path-connected pointed space $(Z,z)$, the set $[S^1,Z]$ is in bijection with the set of conjugacy classes in the fundamental group $\pi_1(Z,z)$ (see \cite[Section 4.A]{Hat}). We will adopt the non-standard notation $[\pi]$ for the set of conjugacy classes in a group $\pi$. It is easily checked that $[S^1, Y_A\times Y_B]\cong [A] \times [B]$, and that $\Delta_\#$ can be identified with the diagonal map $[G] \to [G] \times [G]$ on the conjugacy classes of $G$.

Since $U$ admits a partial section for $q$, there is a map $\psi\co S^1 \to E$ such that $q_\#[\psi]=[\phi]$. Now chasing the element $[\psi]$ around the above diagram, we find that $[\phi]=([a],[b])\in  [A] \times [B]$, where $a\in A$ and $b\in B$ are elements which are conjugate in $G$. By our assumption that $gAg^{-1}\cap B=\{1\}$ for all $g\in G$, this implies that $a=b=1$, and so $\phi$ is in the trivial homotopy class.
\end{proof}

\begin{rem} One could also prove Theorem \ref{main1} by analysing the long exact sequence in homotopy of the fibration $q\co E\to Y_A\times Y_B$ (compare \cite{GLO}). When the condition $AB=G$ does not hold, the total space $E$ is disconnected. The conjugation hypothesis arises from change-of-basepoint isomorphisms. We have chosen the $\cat_1$ proof as more conceptual and easier to follow.
\end{rem}

\bibliographystyle{amsplain}
\bibliography{AsphericalRefs2}

\end{document}